\documentclass[12pt,a4paper]{amsart}
\allowdisplaybreaks[1]
\usepackage{graphicx}

\usepackage[top=30truemm,bottom=30truemm,left=25truemm,right=25truemm]{geometry}

\DeclareFontEncoding{OT2}{}{}
\DeclareFontSubstitution{OT2}{cmr}{m}{l}

\DeclareFontFamily{OT2}{cmr}{\hyphenchar\font45 }
\DeclareFontShape{OT2}{cmr}{m}{l}{%
<5><6><7><8><9>gen*wncyr%
<10><10.95><12><14.4><17.28><20.74><24.88>wncyr10}{}

\DeclareMathAlphabet{\mathcyr}{OT2}{cmr}{m}{l}
\DeclareMathAlphabet{\mathcyb}{OT2}{cmr}{b}{l}
\SetMathAlphabet{\mathcyr}{bold}{OT2}{cmr}{b}{l}

\usepackage{amstext}
\usepackage{amsthm}
\usepackage{amssymb}

\newtheorem{thm}{Theorem}[section]
\newtheorem{lem}[thm]{Lemma}

\theoremstyle{definition}
\newtheorem{defn}[thm]{Definition}

\theoremstyle{remark}
\newtheorem{rem}[thm]{Remark}

\newcommand{\hast}{\mathbin{\hat{\ast}}}
\newcommand{\sha}{\mathbin{\widetilde{\mathcyr{sh}}}}

\begin{document}

\title[A note on Ohno sums for MZVs]{A note on Ohno sums for multiple zeta values}

\author{Hideki Murahara}
\address[Hideki Murahara]{The University of Kitakyushu,
4-2-1 Kitagata, Kokuraminami-ku, Kitakyushu, Fukuoka, 802-8577, Japan} 
\email{hmurahara@mathformula.page}


\begin{abstract}
 The Ohno relation is a well-known relation among multiple zeta values. 
 Hirose, Onozuka, Sato, and the author investigated the sum related to the Ohno relation and presented two types of new relations and five conjectural formulas. 
 This paper proves one of these formulas. 
\end{abstract}

\maketitle

\section{Introduction}
We call a sequence of positive integers ``index'' and an index whose last component greater than $1$ ``admissible index''.
For an admissible index $(k_1,\dots, k_r)$, the multiple zeta values (MZVs) are defined by 
\begin{align*}
 \zeta(k_1,\dots, k_r)
 :=\sum_{1\le m_1<\cdots <m_r} \frac{1}{m_1^{k_1}\cdots m_r^{k_r}} \in \mathbb{R}.  
\end{align*} 
For an index (or a sequence of nonnegative integers) $\boldsymbol{k}=(k_1,\dots,k_r)$, we call $\left|\boldsymbol{k}\right|=k_1+\cdots+k_r$ its weight and $r$ its depth. 
For an admissible index 
\[
 \boldsymbol{k}=(\underbrace{1,\ldots,1}_{a_1-1},b_1+1,\dots,\underbrace{1,\ldots,1}_{a_l-1},b_l+1) \quad (a_p, b_q\ge1),
\]
we define the dual index of $\boldsymbol{k}$ by 
\[
 \boldsymbol{k}^\dagger :=(\underbrace{1,\ldots,1}_{b_l-1},a_l+1,\dots,\underbrace{1,\ldots,1}_{b_1-1},a_1+1).
\]
If an index $\boldsymbol{k}$ and a sequence of nonnegative integers $\boldsymbol{e}$ are the same depth, we denote by $\boldsymbol{k} \oplus \boldsymbol{e}$ the index obtained by componentwise addition. 
Hereafter, if we write $\boldsymbol{e}$, we always assume that $\boldsymbol{e}$ runs over sequences of nonnegative integers having suitable depth. 
In \cite{HMOS20}, Hirose, Onozuka, Sato, and the author considered the following sums.
\begin{defn}[Ohno sum]
 For an admissible index $\boldsymbol{k}$ and a nonnegative integer $m$, 
 we define $\mathcal{O}_m(\boldsymbol{k})$, $\mathcal{O}(\boldsymbol{k})$, respectively by
 \begin{align*}
  \mathcal{O}_m(\boldsymbol{k})
  &:=\sum_{|\boldsymbol{e}|=m}
  \zeta (\boldsymbol{k}\oplus\boldsymbol{e}) \in\mathbb{R}, \\
  \mathcal{O}(\boldsymbol{k})
  &:=\sum_{m=0}^\infty \mathcal{O}_m(\boldsymbol{k})X^m \in\mathbb{R}[[X]].
 \end{align*}
\end{defn}
Using the above notation, the celebrated Ohno relation is written as follows: 
\begin{thm}[Ohno relation; Ohno \cite{Oho99}] \label{ohno}
 For an admissible index $\boldsymbol{k}$, we have
 \begin{align*}
  \mathcal{O} (\boldsymbol{k})
  =\mathcal{O} (\boldsymbol{k}^\dagger).
 \end{align*}
\end{thm}

We denote by $\mathcal{I}$ the $\mathbb{Q}$-linear space spanned by all indices.
We define the $\mathbb{Q}$-bilinear product $\sha \colon \mathcal{I} \times \mathcal{I} \to \mathcal{I}$ by  
\begin{align*}
 \emptyset \sha \boldsymbol{k}&=\boldsymbol{k}\sha \emptyset=\boldsymbol{k}, \\
 (\boldsymbol{k},k) \sha (\boldsymbol{l},l) 
 &= (\boldsymbol{k} \sha (\boldsymbol{l},l), k) 
  +((\boldsymbol{k},k) \sha \boldsymbol{l}, l),
\end{align*}
where $\boldsymbol{k},\boldsymbol{l}\in\mathcal{I}$ and $k,l\in\mathbb{Z}_{\ge1}$.  
For example, we have $(a) \sha (b,c)=(a,b,c)+(b,a,c)+(b,c,a)$. 
For an admissible index $\boldsymbol{k}$ and an integer $s\ge2$, put
 \[
  F(s;\boldsymbol{k})
  :=\mathcal{O}((s)\sha \boldsymbol{k})-\mathcal{O}((s)\sha \boldsymbol{k}^\dagger).
 \]
Then the following Ohno sum relation holds:
\begin{thm}[Hirose--Murahara--Onozuka--Sato \cite{HMOS20}] \label{HMOS}
 For integers $s,t\ge2$, we have
 \[
  F(s;(t+1))=F(t;(s+1)).
 \]
\end{thm}

We denote by $\{a\}^l$ $l$-times repetition of $a$, e.g., $(\{2\}^3)=(2,2,2)$. 
The aim of this paper is to prove the following relation conjectured by Hirose, Onozuka, Sato, and the author in \cite{HMOS20}. 
\begin{thm}[Main theorem] \label{main}
 For integers $s,t\ge2$ and $l\ge0$, we have
 \[
  F(s;(t+1)\sha(\{2\}^{l}))
  =F(t;(s+1)\sha(\{2\}^{l})).
 \]
\end{thm}
\begin{rem}
 The case $l=0$ in Theorem \ref{main} gives Theorem \ref{HMOS}.
\end{rem}

\section{Proof of Theorem \ref{main}}
In the following, everything defined for indices, such as $\zeta$, will be extended $\mathbb{Q}$-linearly. 
We denote by $(k) \hast \boldsymbol{l}$ the formal sum of indices found by adding $k$ to each element of $\boldsymbol{l}$. 
For example, if $\boldsymbol{l}=(l_1,\dots,l_{r'})$, we have 
\[
 (k) \hast \boldsymbol{l}=(l_{1}+k,l_{2},\dots,l_{r'})+\cdots+(l_{1},\dots,l_{r'-1},l_{r'}+k).
\]
Note that, throughout this paper, we often use the special case of  the harmonic product formula  $\zeta(k)\zeta(\boldsymbol{l})=\zeta((k) \ast \boldsymbol{l})$, 
where
\begin{align*}
 (k) \ast \boldsymbol{l}
 =(k) \sha \boldsymbol{l} +(k) \hast \boldsymbol{l}.
\end{align*}

For an admissible index $\boldsymbol{k}$, integers $s,t\ge2$, and nonnegative integers $m,l$, we define
\begin{align*}
 F_{m,l}(s;\boldsymbol{k})
 &:=\mathcal{O}_m((s) \sha \boldsymbol{k} \sha (\{2\}^l) )-\mathcal{O}_m((s) \sha (\boldsymbol{k} \sha (\{2\}^l))^\dagger),\\
 D_{m,l}(s,t)
 &:=F_{m,l}(s;(t+1) )-F_{m,l}(t;(s+1) ). 
\end{align*}
Note that Theorem \ref{main} is equivalent to $D_{m,l}(s,t)=0$ for any $s,t\ge2$ and $m,l\ge0$. 
The main strategy of the proof is similar to \cite{HMOS20}, however, some calculation methods are modified to bear a little more complicated calculations. 
Especially, to show $D_{m,l}(s,2)=0$ in the last part of the proof, we use the Hoffman relation (see Theorem \ref{hof}). 
The proof goes through the following steps: 
\begin{align*}
\begin{array}{lcr}
\begin{array}{lcr}
  &\textrm{Lemmas \ref{Fmpre1} and \ref{Fmpre2}} \Rightarrow \textrm{Lemma \ref{Fm}} \\
  &\textrm{Lemma \ref{oooo}}  
\end{array}
 &\Bigr\} \Rightarrow \textrm{Lemma \ref{dddd}} \\ 
 &\textrm{Theorem \ref{hof}} 
\end{array}
 \Biggr\} \Rightarrow \textrm{Theorem \ref{main}}.
\end{align*}
In the following, we omit the range of $m$'s in the summation if they are integer of $0$ or more, 
for example, we often write $\sum_{\substack{ m_1+\cdots+m_r=m \\ m_1,\dots,m_r\ge0 }}$ as $\sum_{m_1+\cdots+m_r=m}$.

\begin{lem} \label{Fmpre1} 
 For integers $s\ge2$, $t\ge1$, and $m,l\ge0$, we have
 \begin{align*} 
  F_{m,l}(s;(t+1)) 
  &=-\sum_{m_1+m_2=m} \sum_{|\boldsymbol{e}|=m_2} 
   \zeta((s+m_1) \hast ( ((t+1)\sha(\{2\}^l)) \oplus \boldsymbol{e})) \\
  &\quad +\sum_{m_1+m_2=m} \sum_{|\boldsymbol{e}|=m_2} 
   \zeta((s+m_1) \hast ( ((t+1) \sha(\{2\}^l) )^\dagger \oplus \boldsymbol{e})).
 \end{align*}
\end{lem}
\begin{proof}
 By definitions, we have
 \begin{align} \label{AAA}
 \begin{split}
  &\sum_{m_1+m_2=m}\mathcal{O}_{m_1}(s)\mathcal{O}_{m_2}((t+1)\sha(\{2\}^l)) \\
  &=\sum_{m_1+m_2=m} \sum_{|\boldsymbol{e}|=m_2}
   \zeta( (s+m_1)\ast ( ((t+1)\sha(\{2\}^l)) \oplus \boldsymbol{e})) \\
  &=\sum_{m_1+m_2=m} \sum_{|\boldsymbol{e}|=m_2} 
   \zeta((s+m_1) \sha ( ((t+1)\sha(\{2\}^l)) \oplus \boldsymbol{e})) \\
  &\quad +\sum_{m_1+m_2=m} \sum_{|\boldsymbol{e}|=m_2} 
   \zeta((s+m_1) \hast ( ((t+1)\sha(\{2\}^l)) \oplus \boldsymbol{e})) \\
  &=\mathcal{O}_m((s) \sha (t+1) \sha (\{2\}^l)) \\
  &\quad +\sum_{m_1+m_2=m} \sum_{|\boldsymbol{e}|=m_2} 
   \zeta((s+m_1) \hast ( ((t+1)\sha(\{2\}^l)) \oplus \boldsymbol{e})).
 \end{split}
 \end{align}
 We also have
 \begin{align} \label{BBB}
 \begin{split}
  &\sum_{m_1+m_2=m}\mathcal{O}_{m_1}(s)\mathcal{O}_{m_2}((t+1)\sha(\{2\}^l)) \\
  &=\sum_{m_1+m_2=m}\mathcal{O}_{m_1}(s)\mathcal{O}_{m_2}(((t+1)\sha(\{2\}^l))^\dagger) 
   \qquad \textrm{(by Theorem \ref{ohno})} \\
  &=\sum_{m_1+m_2=m}\sum_{|\boldsymbol{e}|=m_2}\zeta((s+m_1) \ast (((t+1)\sha(\{2\}^l))^\dagger \oplus \boldsymbol{e}))\\
  &=\mathcal{O}_m((s) \sha ((t+1) \sha (\{2\}^l))^\dagger ) \\
  &\quad +\sum_{m_1+m_2=m} \sum_{|\boldsymbol{e}|=m_2} 
   \zeta((s+m_1) \hast ( ((t+1) \sha(\{2\}^l) )^\dagger \oplus \boldsymbol{e})).
 \end{split}
 \end{align}
 From \eqref{AAA} and \eqref{BBB}, we have
 \begin{align*}
  &\mathcal{O}_m((s) \sha (t+1) \sha (\{2\}^l)) 
   -\mathcal{O}_m((s) \sha ((t+1) \sha (\{2\}^l))^\dagger ) \\
  &=-\sum_{m_1+m_2=m} \sum_{|\boldsymbol{e}|=m_2} 
   \zeta((s+m_1) \hast ( ((t+1)\sha(\{2\}^l)) \oplus \boldsymbol{e})) \\
  &\quad +\sum_{m_1+m_2=m} \sum_{|\boldsymbol{e}|=m_2} 
   \zeta((s+m_1) \hast ( ((t+1) \sha(\{2\}^l) )^\dagger \oplus \boldsymbol{e})).
 \end{align*}
 By the definition of $F_{m,l}(s;(t+1))$, we obtain the result. 
\end{proof}

\begin{lem} \label{Fmpre2}
 For integers $s,t\ge1$, $m\ge1$, and $l\ge0$, we have
 \begin{align*}
  &\sum_{m_1+m_2=m} \sum_{|\boldsymbol{e}|=m_2}
   \zeta((s+m_1) \hast ( ((t+1)\sha(\{2\}^l)) \oplus \boldsymbol{e})) \\
  &-\sum_{m_1+m_2=m-1} \sum_{|\boldsymbol{e}|=m_2} 
   \zeta((s+m_1+1) \hast ( ((t+1)\sha(\{2\}^l)) \oplus \boldsymbol{e})) \\
  &=\mathcal{O}_{m} ( (s) \hast ( (t+1) \sha (\{2\}^l) ) )
 \end{align*}
 and
 \begin{align*}
  &\sum_{m_1+m_2=m} \sum_{|\boldsymbol{e}|=m_2}
   \zeta((s+m_1) \hast ( ((t+1)\sha(\{2\}^l))^\dagger \oplus \boldsymbol{e}) ) \\
  &-\sum_{m_1+m_2=m-1} \sum_{|\boldsymbol{e}|=m_2} 
   \zeta((s+m_1+1) \hast ( ((t+1)\sha(\{2\}^l))^\dagger \oplus \boldsymbol{e}) ) \\
  &=\mathcal{O}_{m} ( (s) \hast ( (t+1) \sha (\{2\}^l) )^\dagger ). 
 \end{align*} 
\end{lem}
\begin{proof}
 Since
 \begin{align*} 
  \begin{split} 
  &\sum_{m_1+m_2=m} \sum_{|\boldsymbol{e}|=m_2}
   \zeta ((l_{1},\dots,l_{i-1},l_i+m_1,l_{i+1},\dots,l_{r'}) \oplus \boldsymbol{e}) \\
  &-\sum_{m_1+m_2=m-1} \sum_{|\boldsymbol{e}|=m_2}
   \zeta ( (l_{1},\dots,l_{i-1},l_i+m_1+1,l_{i+1},\dots,l_{r'}) \oplus \boldsymbol{e}) \\
  &=\mathcal{O}_{m} (l_{1},\dots,l_{r'})
  \end{split}
 \end{align*}
 holds for an admissible index $(l_{1},\dots,l_{r'})\in\mathcal{I}$ and positive integers $m,i$ with $1\le i\le r'$,
 we have
 \begin{align*}
  &\sum_{m_1+m_2=m} \sum_{|\boldsymbol{e}|=m_2}
   \zeta((k+m_1) \hast ( (l_{1},\dots,l_{r'}) \oplus \boldsymbol{e})) \\
  &-\sum_{m_1+m_2=m-1} \sum_{|\boldsymbol{e}|=m_2} 
   \zeta((k+m_1+1) \hast ( (l_{1},\dots,l_{r'}) \oplus \boldsymbol{e})) \\
  &=\mathcal{O}_{m} ( (k) \hast (l_{1},\dots,l_{r'}) ) 
 \end{align*}
 for a nonnegative integer $k$.
 By using this equality, we obtain the first equality. 
 The second equality can be proved in the same way. 
\end{proof}

\begin{lem} \label{Fm}
 For integers $s\ge3$, $t\ge1$, $m\ge1$, and $l\ge0$, we have
 \begin{align*}
  &F_{m,l}(s-1;(t+1)) -F_{m-1,l}(s;(t+1)) \\
  &=-\mathcal{O}_{m} ( (s-1) \hast ( (t+1) \sha (\{2\}^l) ) ) 
   +\mathcal{O}_{m} ( (s-1) \hast ( (t+1) \sha (\{2\}^l) )^\dagger ).
 \end{align*}
\end{lem}
\begin{proof}
 By Lemmas \ref{Fmpre1} and \ref{Fmpre2}, we have
 \begin{align*}
  &F_{m,l}(s-1;(t+1)) -F_{m-1,l}(s;(t+1)) \\
  &=-\Biggl( \sum_{m_1+m_2=m} \sum_{|\boldsymbol{e}|=m_2} 
   \zeta((s+m_1-1) \hast ( ((t+1)\sha(\{2\}^l)) \oplus \boldsymbol{e})) \\
  &\qquad\quad -\sum_{m_1+m_2=m-1} \sum_{|\boldsymbol{e}|=m_2} 
   \zeta((s+m_1) \hast ( ((t+1)\sha(\{2\}^l)) \oplus \boldsymbol{e})) \Biggr) \\
  &\quad +\Biggl( \sum_{m_1+m_2=m} \sum_{|\boldsymbol{e}|=m_2} 
   \zeta((s+m_1-1) \hast ( ((t+1) \sha(\{2\}^l) )^\dagger \oplus \boldsymbol{e})) \\
  &\qquad\quad -\sum_{m_1+m_2=m-1} \sum_{|\boldsymbol{e}|=m_2} 
   \zeta((s+m_1) \hast ( ((t+1) \sha(\{2\}^l) )^\dagger \oplus \boldsymbol{e})) \Biggr) \\
  &=-\mathcal{O}_{m} ( (s-1) \hast ( (t+1) \sha (\{2\}^l) ) ) 
   +\mathcal{O}_{m} ( (s-1) \hast ( (t+1) \sha (\{2\}^l) )^\dagger ).
  \end{align*}
 This finishes the proof. 
\end{proof}

\begin{lem} \label{oooo}
 For integers $s,t\ge3$ and $l,m\ge0$, we have 
 \begin{align*}
  &\mathcal{O}_{m} ( (s) \sha ( (t) \sha (\{2\}^l) )^\dagger ) 
  -\mathcal{O}_{m} ( (s-1) \hast ( (t+1) \sha (\{2\}^l) )^\dagger ) \\
  &-\bigl( \mathcal{O}_{m} ( (t) \sha ( (s) \sha (\{2\}^l) )^\dagger ) 
   -\mathcal{O}_{m} ( (t-1) \hast ( (s+1) \sha (\{2\}^l) )^\dagger ) \bigr) 
  =0.  
 \end{align*}
\end{lem}
\begin{proof}
 We shall show this lemma only for $l\ge1$. 
 The case $l=0$ can be shown similarly. 
 We have
 \begin{align*}
  &(s) \sha ( (t) \sha (\{2\}^l) )^\dagger 
  =\sum_{i=0}^{l} (s) \sha (\{2\}^{i},t,\{2\}^{l-i})^\dagger \\
  &=\sum_{i=0}^{l} (s) \sha (\{2\}^{i},\{1\}^{t-2},\{2\}^{l-i+1}) \\
  &=\sum_{i=0}^{l} 
   \Biggl( 
    \sum_{j=0}^{i} (\{2\}^{j},s,\{2\}^{i-j},\{1\}^{t-2},\{2\}^{l-i+1}) 
    +\sum_{j=1}^{t-2} (\{2\}^{i},\{1\}^{j},s,\{1\}^{t-j-2},\{2\}^{l-i+1}) \\
    &\qquad\quad +\sum_{j=0}^{l-i} (\{2\}^{i},\{1\}^{t-2},\{2\}^{j+1},s,\{2\}^{l-i-j}) 
   \Biggr)
 \end{align*}
 and
 \begin{align*}
  &(s-1) \hast ( (t+1) \sha (\{2\}^l) )^\dagger 
  =\sum_{i=0}^{l} (s-1) \hast (\{2\}^{i},t+1,\{2\}^{l-i})^\dagger \\
  &=\sum_{i=0}^{l} (s-1) \hast (\{2\}^{i},\{1\}^{t-1},\{2\}^{l-i+1}) \\
  &=\sum_{i=1}^{l} \sum_{j=0}^{i-1} (\{2\}^{j},s+1,\{2\}^{i-j-1},\{1\}^{t-1},\{2\}^{l-i+1}) \\
  &\quad +\sum_{i=0}^{l} 
   \Biggl( 
    \sum_{j=0}^{t-2} (\{2\}^{i},\{1\}^{j},s,\{1\}^{t-j-2},\{2\}^{l-i+1}) 
    +\sum_{j=0}^{l-i} (\{2\}^{i},\{1\}^{t-1},\{2\}^{j},s+1,\{2\}^{l-i-j}) 
   \Biggr).
 \end{align*}
 Thus, we have
 \begin{align*}
  &\mathcal{O}_{m} ( (s) \sha ( (t) \sha (\{2\}^l) )^\dagger ) 
  -\mathcal{O}_{m} ( (s-1) \hast ( (t+1) \sha (\{2\}^l) )^\dagger ) \\
  &=\sum_{i=0}^{l} 
   \Biggl( 
    \sum_{j=0}^{i} \mathcal{O}_{m} (\{2\}^{j},s,\{2\}^{i-j},\{1\}^{t-2},\{2\}^{l-i+1}) 
     +\sum_{j=0}^{l-i} \mathcal{O}_{m} (\{2\}^{i},\{1\}^{t-2},\{2\}^{j+1},s,\{2\}^{l-i-j}) \\
    &\qquad\quad -\sum_{j=0}^{l-i} \mathcal{O}_{m} (\{2\}^{i},\{1\}^{t-1},\{2\}^{j},s+1,\{2\}^{l-i-j}) 
     -\mathcal{O}_{m} (\{2\}^{i},s,\{1\}^{t-2},\{2\}^{l-i+1})
   \Biggr) \\
    &\quad -\sum_{i=1}^{l} \sum_{j=0}^{i-1} \mathcal{O}_{m} (\{2\}^{j},s+1,\{2\}^{i-j-1},\{1\}^{t-1},\{2\}^{l-i+1}) \\
  &=\sum_{\substack{ a+b+c=l \\ a,b,c\ge0 }} 
   \bigl( 
    \mathcal{O}_{m} (\{2\}^{a},s,\{2\}^{b},\{1\}^{t-2},\{2\}^{c+1}) 
    +\mathcal{O}_{m} (\{2\}^{a},\{1\}^{t-2},\{2\}^{b+1},s,\{2\}^{c}) \\
   &\qquad\qquad -\mathcal{O}_{m} (\{2\}^{a},\{1\}^{t-1},\{2\}^{b},s+1,\{2\}^{c})
   \bigr) \\
  &\quad -\sum_{\substack{ a+b=l \\ a,b\ge0 }} 
   \mathcal{O}_{m} (\{2\}^{a},s,\{1\}^{t-2},\{2\}^{b+1}) 
  -\sum_{\substack{ a+b+c=l-1 \\ a,b,c\ge0 }} 
   \mathcal{O}_{m} (\{2\}^{a},s+1,\{2\}^{b},\{1\}^{t-1},\{2\}^{c+1}).
 \end{align*}
 Similarly, we also have
 \begin{align*}
  &\mathcal{O}_{m} ( (t) \sha ( (s) \sha (\{2\}^l) )^\dagger ) 
  -\mathcal{O}_{m} ( (t-1) \hast ( (s+1) \sha (\{2\}^l) )^\dagger ) \\
  &=\sum_{\substack{ a+b+c=l \\ a,b,c\ge0 }} 
   \bigl( 
    \mathcal{O}_{m} (\{2\}^{a},t,\{2\}^{b},\{1\}^{s-2},\{2\}^{c+1}) 
    +\mathcal{O}_{m} (\{2\}^{a},\{1\}^{s-2},\{2\}^{b+1},t,\{2\}^{c}) \\
   &\qquad\qquad -\mathcal{O}_{m} (\{2\}^{a},\{1\}^{s-1},\{2\}^{b},t+1,\{2\}^{c})
   \bigr) \\
  &\quad -\sum_{\substack{ a+b=l \\ a,b\ge0 }} 
   \mathcal{O}_{m} (\{2\}^{a},t,\{1\}^{s-2},\{2\}^{b+1}) 
  -\sum_{\substack{ a+b+c=l-1 \\ a,b,c\ge0 }} 
   \mathcal{O}_{m} (\{2\}^{a},t+1,\{2\}^{b},\{1\}^{s-1},\{2\}^{c+1}).
 \end{align*}
 Then, by using Theorem \ref{ohno}, we find the result. 
\end{proof}

\begin{lem} \label{dddd}
 For integers $s,t\ge3$ and $m\ge1$, we have
 \begin{align*}
  D_{m-1,l}(s,t)=D_{m,l}(s-1,t)+D_{m,l}(s,t-1).
 \end{align*}
\end{lem}
\begin{proof} 
 By definitions, we have
 \begin{align*}
  &D_{m-1,l}(s,t) -D_{m,l}(s-1,t) -D_{m,l}(s,t-1) \\
  &=F_{m-1,l}(s;(t+1))-F_{m-1,l}(t;(s+1)) \\
   &\quad -(F_{m,l}(s-1;(t+1))-F_{m,l}(t;(s)) ) 
    -(F_{m,l}(s;(t) )-F_{m,l}(t-1;(s+1)) ) \\
  &=-( F_{m,l}(s-1;(t+1)) -F_{m-1,l}(s;(t+1)) ) 
   +( F_{m,l}(t-1;(s+1)) -F_{m-1,l}(t;(s+1)) ) \\ 
   &\quad +F_{m,l}(t;(s)) -F_{m,l}(s;(t)).
 \end{align*}
 Then, by Lemma \ref{Fm}, we have 
 \begin{align*}
  &D_{m-1,l}(s,t) -D_{m,l}(s-1,t) -D_{m,l}(s,t-1) \\
  &=\mathcal{O}_{m} ( (s-1) \hast ( (t+1) \sha (\{2\}^l) ) ) 
    -\mathcal{O}_{m} ( (s-1) \hast ( (t+1) \sha (\{2\}^l) )^\dagger ) \\
   &\quad -\mathcal{O}_{m} ( (t-1) \hast ( (s+1) \sha (\{2\}^l) ) ) 
    +\mathcal{O}_{m} ( (t-1) \hast ( (s+1) \sha (\{2\}^l) )^\dagger ) \\
   &\quad +\mathcal{O}_{m} ( (s) \sha ( (t) \sha (\{2\}^l) )^\dagger ) 
    -\mathcal{O}_{m} ( (t) \sha ( (s) \sha (\{2\}^l) )^\dagger ).  
 \end{align*}
 Since
 \begin{align*}
  &\mathcal{O}_{m} ( (s-1) \hast ( (t+1) \sha (\{2\}^l) ) )
  -\mathcal{O}_{m} ( (t-1) \hast ( (s+1) \sha (\{2\}^l) ) ) \\
  &=\mathcal{O}_{m} ( (s+t) \sha (\{2\}^l) +(s+1) \sha (t+1) \sha (\{2\}^{l-1}) ) \\
   &\quad -\mathcal{O}_{m} ( (s+t) \sha (\{2\}^l) +(s+1) \sha (t+1) \sha (\{2\}^{l-1}) )
   =0
 \end{align*}
 and by Lemma \ref{oooo}, we obtain the result. 
\end{proof}

In the proof of Theorem \ref{main}, we use the following theorem. 
\begin{thm}[Hoffman {\cite[Theorem 5.1]{Hof92}}] \label{hof}
 For positive integers $k_{1},\dots,k_{r}$ with $k_r\ge2$, we have
 \begin{align*}
  &\sum_{i=1}^{r} \zeta(k_{1},\dots,k_{i-1},k_i+1,k_{i+1},\dots,k_{r}) \\
  &=\sum_{\substack{ 1\le i\le r \\ k_i\ge2 }} \sum_{j=0}^{k_i-2} 
   \zeta(k_{1},\dots,k_{i-1},j+1,k_i-j,k_{i+1},\dots,k_{r}).
 \end{align*}
\end{thm}

Now we prove our main theorem. 
\begin{proof}[Proof of Theorem \ref{main}]
 By Lemma \ref{dddd} and the identity $D_{m,l}(s,t)=-D_{m,l}(t,s)$, we need to show only the case $D_{m,l}(s,2)=0$, i.e., 
 $F_{m,l}(s;(3))=F_{m,l}(2;(s+1))$. 
 By Lemma \ref{Fmpre1}, we have
 \begin{align*}
  &F_{m,l}(s;(3)) -F_{m,l}(2;(s+1)) \\
  &=\sum_{m_1+m_2=m} \sum_{|\boldsymbol{e}|=m_2} 
   \bigl(
    -\zeta((s+m_1) \hast ( ((3)\sha(\{2\}^l)) \oplus \boldsymbol{e})) \\
    &\qquad\qquad\qquad\qquad +\zeta((s+m_1) \hast ( ((3) \sha(\{2\}^l) )^\dagger \oplus \boldsymbol{e})) 
   \bigr) \\
  &\quad -(l+1) \mathcal{O}_{m} ( (s+1) \sha \{ 2 \}^{l+1} ) 
   +\mathcal{O}_{m} ( (2) \sha ((s+1) \sha \{ 2 \}^{l})^\dagger ).
 \end{align*}
 Put
 \begin{align*}
  A&:=-\sum_{m_1+m_2=m} \sum_{|\boldsymbol{e}|=m_2} 
    \zeta((s+m_1) \hast ( ((3)\sha(\{2\}^l)) \oplus \boldsymbol{e})), \\
  B&:=\sum_{m_1+m_2=m} \sum_{|\boldsymbol{e}|=m_2}
    \zeta((s+m_1) \hast ( ((3) \sha(\{2\}^l) )^\dagger \oplus \boldsymbol{e})), \\
  C&:=-(l+1) \mathcal{O}_{m} ( (s+1) \sha \{ 2 \}^{l+1} ) 
   +\mathcal{O}_{m} ( (2) \sha ((s+1) \sha \{ 2 \}^{l})^\dagger ). 
 \end{align*}
 Then, by definitions, we have
 \begin{align} \label{ppppa}
  A=-\sum_{a=0}^{m} \bigl( 
    \mathcal{O}_{m-a} ( (s+a+3) \sha (\{ 2 \}^{l}) ) 
    +\mathcal{O}_{m-a} ( (s+a+2) \sha (3) \sha (\{ 2 \}^{l-1}) ) 
   \bigr), 
 \end{align}
 where we understand the second term is $0$ if $l=0$, and the same shall apply hereinafter. 
 By Lemma \ref{add1} in the next section, we get
 \begin{align} \label{ppppp}
  \begin{split}
  A&=-\sum_{m_1+\cdots+m_{l+1}=m+s}
   \Biggl( \sum_{p=1}^{l+1} \max\{ m_p-s+1, 0 \} \Biggr)  \\
   &\qquad\qquad\quad \times \sum_{i=1}^{l+1} \zeta(m_{1}+2,\dots,m_{i-1}+2,m_i+3,m_{i+1}+2,\dots,m_{l+1}+2).
  \end{split}
 \end{align} 

 On the other hand, we have
 \begin{align*}
  B&=\sum_{m_1+m_2=m} 
   \Biggl(  
    \sum_{i=1}^{l} \sum_{j=0}^{i-1}
    \mathcal{O}_{m_{2}} ( \{ 2 \}^{j}, s+m_1+2, \{ 2 \}^{i-j-1}, 1, \{ 2 \}^{l-i+1} ) \\
    &\qquad\qquad\quad\,\, +\sum_{i=0}^{l} 
    \mathcal{O}_{m_{2}} ( \{ 2 \}^{i}, s+m_1+1, \{ 2 \}^{l-i+1} ) \\
    &\qquad\qquad\quad\,\, +\sum_{i=0}^{l} \sum_{j=0}^{l-i}
    \mathcal{O}_{m_{2}} ( \{ 2 \}^{i}, 1, \{ 2 \}^{j}, s+m_1+2, \{ 2 \}^{l-i-j} ) 
   \Biggr) \\
  &=\sum_{a=0}^{m} 
   \bigl( 
    \mathcal{O}_{m-a} ( (1) \sha (s+a+2) \sha (\{ 2 \}^{l-1}), 2 ) \\ 
    &\qquad\quad +\mathcal{O}_{m-a} ( (s+a+1) \sha (\{ 2 \}^{l}), 2 ) 
     +\mathcal{O}_{m-a} ( (1) \sha (\{ 2 \}^{l}), s+a+2 )
   \bigr).
 \end{align*}
 Since
 \begin{align*}
  \mathcal{O}_{m} ( (s+1) \sha (\{ 2 \}^{l+1}) ) 
  =\sum_{i=0}^{l+1} 
   \mathcal{O}_{m} ( \{ 2 \}^{i}, s+1, \{ 2 \}^{l-i+1} ) 
 \end{align*}
 and
 \begin{align*}
  &\mathcal{O}_{m} ( (2) \sha ((s+1) \sha (\{ 2 \}^{l}) )^\dagger ) \\
  &=(l+2) \sum_{i=0}^{l+1} 
   \mathcal{O}_{m} ( \{ 2 \}^{i}, s+1, \{ 2 \}^{l-i+1} ) -\mathcal{O}_{m} (s+1,\{ 2 \}^{l+1}) \\
  &\quad +\sum_{i=0}^{l} \sum_{j=1}^{s-2} 
   \mathcal{O}_{m} ( \{ 2 \}^{i}, j+2, s-j+1, \{ 2 \}^{l-i} ), 
 \end{align*}
 we also have
 \begin{align*}
  C&=\sum_{i=0}^{l+1} 
   \mathcal{O}_{m} ( \{ 2 \}^{i}, s+1, \{ 2 \}^{l-i+1} ) -\mathcal{O}_{m} (s+1,\{ 2 \}^{l+1}) \\
  &\quad +\sum_{i=0}^{l} \sum_{j=1}^{s-2} 
   \mathcal{O}_{m} ( \{ 2 \}^{i}, j+2, s-j+1, \{ 2 \}^{l-i} ) \\
  &=\mathcal{O}_{m} ( 2, (s+1) \sha (\{ 2 \}^{l}) ) 
   +\sum_{i=0}^{l} \sum_{j=1}^{s-2} 
   \mathcal{O}_{m} ( \{ 2 \}^{i}, j+2, s-j+1,  \{ 2 \}^{l-i} ) \\
  &=\sum_{i=0}^{l} \sum_{j=0}^{s-2} 
   \mathcal{O}_{m} ( \{ 2 \}^{i}, j+2, s-j+1,  \{ 2 \}^{l-i} ). 
 \end{align*}
 Then we find
 \begin{align}  \label{qqqqq-}
  \begin{split}
  B+C&=\sum_{a=0}^{m} 
   \bigl( 
    \mathcal{O}_{m-a} ( (1) \sha (s+a+2) \sha (\{ 2 \}^{l-1}), 2 ) \\ 
    &\qquad\quad +\mathcal{O}_{m-a} ( (s+a+1) \sha (\{ 2 \}^{l}), 2 ) 
     +\mathcal{O}_{m-a} ( (1) \sha (\{ 2 \}^{l}), s+a+2 )
   \bigr) \\
  &\quad +\sum_{i=0}^{l} \sum_{j=0}^{s-2} 
   \mathcal{O}_{m} ( \{ 2 \}^{i}, j+2, s-j+1,  \{ 2 \}^{l-i} ). 
 \end{split}
 \end{align}
 By Lemma \ref{add2}, we get
 \begin{align} \label{qqqqq}
  \begin{split}
   &B+C \\
   &=\sum_{m_1+\cdots+m_{l+1}=m+s} 
   \Biggl( \sum_{p=1}^{l+1} \max\{ m_p-s+1, 0 \} \Biggr) \\
   &\qquad \times \sum_{i=1}^{l+1} \sum_{j=0}^{m_i} \zeta(m_{1}+2,\dots,m_{i-1}+2,j+1,m_i-j+2,m_{i+1}+2,\dots,m_{l+1}+2). 
  \end{split}
 \end{align}

 Hence, from \eqref{ppppp} and \eqref{qqqqq}, we find
 \begin{align*}
  &F_{m,l}(s;(3)) -F_{m,l}(2;(s+1)) 
  =A+B+C \\
  &=\sum_{m_1+\cdots+m_{l+1}=m+s}
   \Biggl( \sum_{p=1}^{l+1} \max\{ m_p-s+1, 0 \} \Biggr)  \\
   &\qquad\qquad \times \sum_{i=1}^{l+1} \Biggl( 
    -\zeta(m_{1}+2,\dots,m_{i-1}+2,m_i+3,m_{i+1}+2,\dots,m_{l+1}+2) \\
   &\qquad\qquad\qquad\quad +\sum_{j=0}^{m_i} \zeta(m_{1}+2,\dots,m_{i-1}+2,j+1,m_i-j+2,m_{i+1}+2,\dots,m_{l+1}+2) 
    \Biggr). 
 \end{align*}
 By Theorem \ref{hof}, we obtain the result. 
\end{proof}

\section{Additional lemmas}
Before ending this paper, we shall state some lemmas which were used in the proof of Theorem \ref{main}.
Let $l$ be a positive integer. 
For integers $p,q$ with $1\le p,q\le l+1$, put 
\begin{align*}
 G_{p,q}:=
 \begin{cases}
  \displaystyle{
  \sum_{a=0}^{m} \mathcal{O}_{m-a} (\{ 2 \}^{p-1}, s+a+2, \{ 2 \}^{q-p-1}, 3, \{ 2 \}^{l-q+1}) \quad \quad\, (p<q), }\\
  \displaystyle{
  \sum_{a=0}^{m} \mathcal{O}_{m-a} (\{ 2 \}^{p-1}, s+a+3, \{ 2 \}^{l-p+1}) \qquad \qquad \qquad \quad \, (p=q), }\\
  \displaystyle{
  \sum_{a=0}^{m} \mathcal{O}_{m-a} (\{ 2 \}^{q-1}, 3, \{ 2 \}^{p-q-1}, s+a+2, \{ 2 \}^{l-p+1}) \qquad (p>q) }
 \end{cases}
\end{align*}
and
\begin{align*}
 H_{p,q}
 &:=\sum_{m_1+\cdots+m_l=m+s} \max\{ m_p-s+1 ,0 \} \\
  &\quad \times \zeta(m_{1}+2,\dots,m_{q-1}+2,m_q+3,m_{q+1}+2,\dots,m_{l+1}+2). 
\end{align*}
Note that we have
\begin{align*}
 \textrm{R.H.S. of \eqref{ppppa}}
 &=-\sum_{1\le p,q\le l+1} G_{p,q}, \\
 \textrm{R.H.S. of \eqref{ppppp}}
 &=-\sum_{1\le p,q\le l+1} H_{p,q}.
\end{align*}
\begin{lem} \label{add1}
 For integers $i,j$ with $1\le i,j\le l+1$, we have
 \begin{align*}
  G_{p,q}=H_{p,q}. 
 \end{align*} 
\end{lem}
\begin{proof}
 We prove this lemma only for $i<j$. The other cases can be proved similarly. 

 We have
 \begin{align*}
  G_{p,q}
  &=\sum_{a=0}^{m} \sum_{\substack{ e_1+\cdots+e_{l+1}=m-a \\ e_1,\dots,e_{l+1}\ge0 }}
   \zeta(2+e_1,\dots, 2+e_{p-1}, s+a+2+e_p , \\
   &\qquad \qquad \qquad \qquad \qquad  2+e_{p+1},\dots, 2+e_{q-1}, 3+e_q, 2+e_{q+1},\dots, 2+e_{l+1}) \\
  &=\sum_{a=0}^{m} \sum_{e_p=0}^{m-a} 
   \sum_{\substack{ e_1+\cdots+\hat{e_{p}}+\cdots+e_{l+1}=m-a-e_p \\ e_1,\dots, \hat{e_{p}},\dots, e_{l+1}\ge0 }}
   \zeta(2+e_1,\dots, 2+e_{p-1}, s+a+2+e_p , \\
   &\qquad \qquad \qquad \qquad \qquad  2+e_{p+1},\dots, 2+e_{q-1}, 3+e_q, 2+e_{q+1},\dots, 2+e_{l+1}) \\
  &=\sum_{a=0}^{m} \sum_{u=a}^{m} 
   \sum_{\substack{ e_1+\cdots+\hat{e_{p}}+\cdots+e_{l+1}=m-u \\ e_1,\dots, \hat{e_{p}},\dots, e_{l+1}\ge0 }}
   \zeta(2+e_1,\dots, 2+e_{p-1}, s+u+2, \\
   &\qquad \qquad \qquad \qquad \qquad  2+e_{p+1},\dots, 2+e_{q-1}, 3+e_q, 2+e_{q+1},\dots, 2+e_{l+1}),
 \end{align*}
 where the symbol $\hat{e_p}$ in the summation means that the part $e_p$ is skipped. 
 Then we find
 \begin{align*}
  G_{p,q}
  &=\sum_{u=0}^{m} \sum_{a=0}^{u} 
   \sum_{\substack{ e_1+\cdots+\hat{e_{p}}+\cdots+e_{l+1}=m-u \\ e_1,\dots, \hat{e_{p}},\dots, e_{l+1}\ge0 }}
   \zeta(2+e_1,\dots, 2+e_{p-1}, s+u+2, \\
   &\qquad \qquad \qquad \qquad \qquad  2+e_{p+1},\dots, 2+e_{q-1}, 3+e_q, 2+e_{q+1},\dots, 2+e_{l+1})\\
  &=\sum_{u=0}^{m} (u+1) 
   \sum_{\substack{ e_1+\cdots+\hat{e_{p}}+\cdots+e_{l+1}=m-u \\ e_1,\dots, \hat{e_{p}},\dots, e_{l+1}\ge0 }}
   \zeta(2+e_1,\dots, 2+e_{p-1}, s+u+2, \\
   &\qquad \qquad \qquad \qquad \qquad  2+e_{p+1},\dots, 2+e_{q-1}, 3+e_q, 2+e_{q+1},\dots, 2+e_{l+1}).
 \end{align*}

 On the other hand, we have
 \begin{align*}
  H_{p,q}
  &=\sum_{u=1}^{m+1} u
   \sum_{m_1+\cdots+\hat{m_p}+\cdots+m_{l+1}=m-u+1} 
   \zeta(m_{1}+2,\dots, m_{p-1}+2, s+u+1, \\
   &\qquad \qquad \qquad \qquad m_{p+1}+2, \dots,m_{q-1}+2,m_q+3,m_{q+1}+2,\dots,m_{l+1}+2)\\ 
  &=\sum_{u=0}^{m} (u+1)
   \sum_{m_1+\cdots+\hat{m_p}+\cdots+m_{l+1}=m-u} 
   \zeta(m_{1}+2,\dots, m_{p-1}+2, s+u+2, \\
   &\qquad \qquad \qquad \qquad m_{p+1}+2, \dots,m_{q-1}+2,m_q+3,m_{q+1}+2,\dots,m_{l+1}+2). 
 \end{align*}
 This finishes the proof. 
\end{proof}

For integers $p,q$ with $1\le p,q\le l+1$, put 
\begin{align*}
 I_{p,q}:=
 \begin{cases}
  \displaystyle{ 
  \sum_{a=0}^{m} \mathcal{O}_{m-a} (\{ 2 \}^{p-1}, s+a+2, \{ 2 \}^{q-p-1}, 1, \{ 2 \}^{l-q+2}) 
  \qquad \qquad \qquad\qquad \quad \,\,\, (p<q), }\\
  \displaystyle{ 
  \sum_{a=0}^{m} \Bigl( \mathcal{O}_{m-a} (\{ 2 \}^{p-1}, 1, s+a+2, \{ 2 \}^{l-p+1}) 
    +\mathcal{O}_{m-a} (\{ 2 \}^{p-1}, s+a+1, \{ 2 \}^{l-p+2}) \Bigr) }\\ 
  \displaystyle{   
  \quad +\sum_{j=0}^{s-2} \mathcal{O}_{m} (\{ 2 \}^{p-1}, j+2, s-j+1, \{ 2 \}^{l-p+1}) 
  \qquad \qquad \qquad\qquad \qquad \quad (p=q), }\\
  \displaystyle{ 
  \sum_{a=0}^{m} \mathcal{O}_{m-a} (\{ 2 \}^{q-1}, 1, \{ 2 \}^{p-q}, s+a+2, \{ 2 \}^{l-p+1}) 
  \qquad \qquad \qquad\qquad \qquad \,\,  (p>q) }
 \end{cases}
\end{align*}
and
\begin{align*}
 J_{p,q}
 &:=\sum_{m_1+\cdots+m_{l+1}=m+s} \max\{ m_p-s+1 ,0 \} \\
  &\qquad \times \sum_{j=0}^{m_q} \zeta(m_{1}+2,\dots,m_{q-1}+2,j+1,m_q-j+2,m_{q+1}+2,\dots,m_{l+1}+2). 
\end{align*}
Note that we have
\begin{align*}
 \textrm{R.H.S. of \eqref{qqqqq-}}
 &=\sum_{1\le p,q\le l+1} I_{p,q}, \\
 \textrm{R.H.S. of \eqref{qqqqq}}
 &=\sum_{1\le p,q\le l+1} J_{p,q}.
\end{align*}
\begin{lem} \label{add2}
 For an integer $p,q$ with $1\le p,q\le l+1$, we have
 \begin{align*}
  I_{p,q}=J_{p,q}. 
 \end{align*} 
\end{lem}
\begin{proof}
 The case $p<q$ can be proved in a similar mannar to the proof of the previous lemma.  
 Indeed, we have
 \begin{align*}
  I_{p,q}
  &=\sum_{a=0}^{m} \sum_{\substack{ e_1+\cdots+e_{l+2}=m-a \\ e_1,\dots,e_{l+2}\ge0 }}
   \zeta(2+e_1,\dots, 2+e_{p-1}, s+a+2+e_p , \\
   &\qquad \qquad \qquad \qquad \qquad  2+e_{p+1},\dots, 2+e_{q-1}, 1+e_q, 2+e_{q+1},\dots, 2+e_{l+2}) \\
  &=\sum_{a=0}^{m} \sum_{e_p=0}^{m-a} 
   \sum_{\substack{ e_1+\cdots+\hat{e_{p}}+\cdots+e_{l+2}=m-a-e_p \\ e_1,\dots, \hat{e_{p}},\dots, e_{l+2}\ge0 }}
   \zeta(2+e_1,\dots, 2+e_{p-1}, s+a+2+e_p , \\
   &\qquad \qquad \qquad \qquad \qquad  2+e_{p+1},\dots, 2+e_{q-1}, 1+e_q, 2+e_{q+1},\dots, 2+e_{l+2}) \\
  &=\sum_{a=0}^{m} \sum_{u=a}^{m} 
   \sum_{\substack{ e_1+\cdots+\hat{e_{p}}+\cdots+e_{l+2}=m-u \\ e_1,\dots, \hat{e_{p}},\dots, e_{l+2}\ge0 }}
   \zeta(2+e_1,\dots, 2+e_{p-1}, s+u+2, \\
   &\qquad \qquad \qquad \qquad \qquad  2+e_{p+1},\dots, 2+e_{q-1}, 1+e_q, 2+e_{q+1},\dots, 2+e_{l+2}).
 \end{align*}
 Then we find
 \begin{align*}
  I_{p,q}
  &=\sum_{u=0}^{m} \sum_{a=0}^{u} 
   \sum_{\substack{ e_1+\cdots+\hat{e_{p}}+\cdots+e_{l+2}=m-u \\ e_1,\dots, \hat{e_{p}},\dots, e_{l+2}\ge0 }}
   \zeta(2+e_1,\dots, 2+e_{p-1}, s+u+2, \\
   &\qquad \qquad \qquad \qquad \qquad  2+e_{p+1},\dots, 2+e_{q-1}, 1+e_q, 2+e_{q+1},\dots, 2+e_{l+2})\\
  &=\sum_{u=0}^{m} (u+1) 
   \sum_{\substack{ e_1+\cdots+\hat{e_{p}}+\cdots+e_{l+2}=m-u \\ e_1,\dots, \hat{e_{p}},\dots, e_{l+2}\ge0 }}
   \zeta(2+e_1,\dots, 2+e_{p-1}, s+u+2, \\
   &\qquad \qquad \qquad \qquad \qquad  2+e_{p+1},\dots, 2+e_{q-1}, 1+e_q, 2+e_{q+1},\dots, 2+e_{l+2}).
 \end{align*}
 We also have
 \begin{align*}
 J_{p,q}
 &=\sum_{u=1}^{m+1} u
  \sum_{m_1+\cdots+\hat{m_p}+\cdots+m_{l+1}=m-u+1} 
   \sum_{j=0}^{m_q} \\
   &\qquad \qquad  \zeta(m_{1}+2,\dots,m_{p-1}+2,s+u+1,m_{p+1}+2, \\
   &\qquad \qquad \qquad \dots,m_{q-1}+2,j+1,m_q-j+2,m_{q+1}+2,\dots,m_{l+1}+2) \\ 
 &=\sum_{u=0}^{m} (u+1)
  \sum_{m_1+\cdots+\hat{m_p}+\cdots+m_{l+1}=m-u} 
  \sum_{j=0}^{m_q} \\
   &\qquad \qquad  \zeta(m_{1}+2,\dots,m_{p-1}+2,s+u+2,m_{p+1}+2, \\
   &\qquad \qquad \qquad \dots,m_{q-1}+2,j+1,m_q-j+2,m_{q+1}+2,\dots,m_{l+1}+2) \\
  &=\sum_{u=0}^{m} (u+1) 
  \sum_{m_1+\cdots+\hat{m_p}+\cdots+m_{l+1}=m-u} 
   \zeta(m_1+2,\dots, m_{p-1}+2, s+u+2, \\
   &\qquad \qquad \qquad \qquad \qquad m_{p+1}+2,\dots, m_{q-1}+2, m_q+1, m_{q+1}+2,\dots, m_{l+2}+2).    
 \end{align*}
 Thus we find the case $p<q$ holds, and the case $p>q$ can be proved similarly.  
 
 Now, we consider the case $p=q$. 
 By the similar argument above, we find
 \begin{align*}
  &\sum_{a=0}^{m} \mathcal{O}_{m-a} (\{ 2 \}^{p-1}, 1, s+a+2, \{ 2 \}^{l-p+1}) \\
  &=\sum_{u=0}^{m} (u+1) 
  \sum_{m_1+\cdots+\hat{m_{p+1}}+\cdots+m_{l+2}=m-u} \\
   &\qquad \qquad \quad 
   \zeta(m_1+2,\dots, m_{p-1}+2, m_p+1, s+u+2, m_{p+2}+2,\dots, m_{l+2}+2). 
 \end{align*}
 Thus we have
 \begin{align*}
  &\sum_{a=0}^{m} \mathcal{O}_{m-a} (\{ 2 \}^{p-1}, 1, s+a+2, \{ 2 \}^{l-p+1}) \\
  &=\sum_{u=0}^{m} (u+1) 
  \sum_{v=0}^{m-u} 
  \sum_{m_1+\cdots+m_{l}=m-u-v} \\
   &\qquad \qquad \quad 
   \zeta(m_1+2,\dots, m_{p-1}+2, v+1, s+u+2, m_{p}+2,\dots, m_{l}+2) \\
  &=\sum_{u=0}^{m} (u+1) 
  \sum_{v=u}^{m} 
  \sum_{m_1+\cdots+m_{l}=m-v} \\
   &\qquad \qquad \quad 
   \zeta(m_1+2,\dots, m_{p-1}+2, v-u+1, s+u+2, m_{p}+2,\dots, m_{l}+2) \\
  &=\sum_{v=0}^{m}  
  \sum_{m_1+\cdots+m_{l}=m-v} 
  \sum_{u=0}^{v} (u+1)\\
   &\qquad \qquad \quad 
   \zeta(m_1+2,\dots, m_{p-1}+2, v-u+1, s+u+2, m_{p}+2,\dots, m_{l}+2).   
 \end{align*}  
 Since
 \begin{align*}
  \sum_{u=0}^{v} (u+1) \times 
   (v-u+1, s+u+2) 
  =\sum_{u=0}^{v}
   \sum_{j=1}^{v-u+1} 
   (j, s+v-j+3),   
 \end{align*}
 we get
 \begin{align} \label{g1}
 \begin{split}
  &\sum_{a=0}^{m} \mathcal{O}_{m-a} (\{ 2 \}^{p-1}, 1, s+a+2, \{ 2 \}^{l-p+1}) \\
  &=\sum_{v=0}^{m}  
   \sum_{m_1+\cdots+m_{l}=m-v} 
   \sum_{u=0}^{v} 
   \sum_{j=1}^{v-u+1} \\
    &\qquad \qquad \quad 
    \zeta(m_1+2,\dots, m_{p-1}+2, j, s+v-j+3, m_{p}+2,\dots, m_{l}+2).  
 \end{split}
 \end{align}
 Similarly, we also have        
 \begin{align} \label{g2}
 \begin{split}
  &\sum_{a=0}^{m} \mathcal{O}_{m-a} (\{ 2 \}^{p-1}, s+a+1, \{ 2 \}^{l-p+2}) \\
  &=\sum_{v=0}^{m}  
   \sum_{m_1+\cdots+m_{l}=m-v} 
   \sum_{u=0}^{v} 
   \sum_{j=s+v-u+1}^{s+v+1} \\
    &\qquad \qquad \quad 
    \zeta(m_1+2,\dots, m_{p-1}+2, j, s+v-j+3, m_{p}+2,\dots, m_{l}+2). 
 \end{split}
 \end{align}
 Since
 \begin{align*} 
 &\sum_{j=0}^{s-2} \mathcal{O}_{m} (\{ 2 \}^{p-1}, j+2, s-j+1, \{ 2 \}^{l-p+1}) \\
  &=\sum_{v=0}^{m}  
   \sum_{m_1+\cdots+m_{l}=m-v} 
   \sum_{u=0}^{v} 
   \sum_{j=0}^{s-2} \\
    &\qquad \quad 
    \zeta(m_1+2,\dots, m_{p-1}+2, j-u+v+2, s+u-j+1, m_{p}+2,\dots, m_{l}+2),
 \end{align*}
 we have
 \begin{align} \label{g3}
 \begin{split}
 &\sum_{j=0}^{s-2} \mathcal{O}_{m} (\{ 2 \}^{p-1}, j+2, s-j+1, \{ 2 \}^{l-p+1}) \\
  &=\sum_{v=0}^{m}  
   \sum_{m_1+\cdots+m_{l}=m-v} 
   \sum_{u=0}^{v} 
   \sum_{j=v-u+2}^{s+v-u} \\
    &\qquad \qquad \quad 
    \zeta(m_1+2,\dots, m_{p-1}+2, j, s+v-j+3, m_{p}+2,\dots, m_{l}+2).  
 \end{split}
 \end{align}
 From \eqref{g1}, \eqref{g2}, and \eqref{g3}, we get
 \begin{align*}
  I_{p,p} 
  &=\sum_{v=0}^{m}  
   \sum_{m_1+\cdots+m_{l}=m-v} 
   \sum_{u=0}^{v} 
   \sum_{j=1}^{s+v+1} \\
    &\qquad \qquad \quad 
    \zeta(m_1+2,\dots, m_{p-1}+2, j, s+v-j+3, m_{p}+2,\dots, m_{l}+2)\\
  &=\sum_{v=0}^{m} 
   (v+1) 
   \sum_{m_1+\cdots+m_{l}=m-v} 
   \sum_{j=1}^{s+v+1} \\
    &\qquad \qquad \quad 
    \zeta(m_1+2,\dots, m_{p-1}+2, j, s+v-j+3, m_{p}+2,\dots, m_{l}+2).    
 \end{align*}

 Hence, we have $I_{p,p}=J_{p,p}$.
 This finishes the proof. 
\end{proof}


\end{document}